\documentclass[12pt]{amsart}
\usepackage{amssymb,amscd}
\usepackage{amsmath}

\begin{document}
\newtheorem{thm}{Theorem}[section]
\newtheorem{lem}[thm]{Lemma}
\newtheorem{prop}[thm]{Proposition}
\newtheorem{cor}[thm]{Corollary}
\theoremstyle{definition}
\newtheorem{ex}[thm]{Example}
\newtheorem{rem}[thm]{Remark}
\newtheorem{prob}[thm]{Problem}
\newtheorem{thmA}{Theorem}
\renewcommand{\thethmA}{}
\newtheorem{defi}[thm]{Definition}
\renewcommand{\thedefi}{}
\input amssym.def
\long\def\alert#1{\smallskip{\hskip\parindent\vrule%
\vbox{\advance\hsize-2\parindent\hrule\smallskip\parindent.4\parindent%
\narrower\noindent#1\smallskip\hrule}\vrule\hfill}\smallskip}
\def\ff{\frak}
\def\Spec{\mbox{\rm Spec}}
\def\type{\mbox{ type}}
\def\Hom{\mbox{ Hom}}
\def\rank{\mbox{ rank}}
\def\Ext{\mbox{ Ext}}
\def\Ker{\mbox{ Ker}}
\def\Max{\mbox{\rm Max}}
\def\End{\mbox{\rm End}}
\def\l{\langle\:}
\def\r{\:\rangle}
\def\Rad{\mbox{\rm Rad}}
\def\Zar{\mbox{\rm Zar}}
\def\Supp{\mbox{\rm Supp}}
\def\Rep{\mbox{\rm Rep}}
\def\cal{\mathcal}
\title[Rings whose ideals form a BL-algebra]{BL-rings}
\thanks{2010 Mathematics Subject Classification.
06D35, 06E15, 06D50}
\thanks{\today}
\author{Olivier A, Heubo-Kwegna, Celestin Lele, Jean B. Nganou}
\address{Department of Mathematical Sciences, Saginaw Valley StateUniversity,
Cameroon} \email{oheubokw@svsu.edu}
\address{Department of Mathematics and Computer Science, University of Dschang, Cameroon
} \email{celestinlele@yahoo.com}
\address{Department of Mathematics, University of Oregon, Eugene,
OR 97403} \email{nganou@uoregon.edu}
\begin{abstract} The main goal of this article is to introduce BL-rings, i.e., commutative rings whose lattices of ideals can be equipped with a structure of BL-algebra. We obtain a description of such rings, and study the connections between the new class and well known classes such as multiplications rings, Baer rings, Dedekind rings.
\vspace{0.10in}\\
{\noindent} Key words: Multiplication ring, Baer ring, subdirectly irreducible ring, BL-ring.
\end{abstract}
\maketitle
\section{Introduction}
Given any ring (commutative or not, with or without unity) $R$ generated by idempotents, the semiring of ideals of $R$ under the usual operations form a residuated lattice $A(R)$. In recent articles, several authors have investigated classes of rings for which the residuated lattice $A(R)$ an algebra of a well-known subvariety of residuated lattices. For instance, rings $R$ for which $A(R)$ is an MV-algebra, also called \L ukasiewicz rings are investigated in \cite{BN}, rings $R$ for which $A(R)$ is a G\"{o}del algebra, also called  G\"{o}del rings are investigated in \cite{BNM}, and very recently rings $R$ for which $A(R)$ is an pseudo MV-algebra, also called Generalized \L ukasiewicz rings are investigated in \cite{KLN}.

In the same spirit, the goal of the present article is we introduce and investigate the class of commutative rings $R$ for which $A(R)$ is a BL-algebra, also referred to as BL-rings. Among other things, we show that this class is properly contained in the class of multiplication rings as treated in \cite{Grif}, and contains properly each of the classes of Dedekind domains, \L ukasiewicz rings, discrete valuation rings, Noetherian multiplication rings. We also prove that BL-rings are closed under finite direct products, arbitrary direct sums, and homomorphic images. Furthermore, a description of subdirectly irreducible BL-rings is obtained, which combined with the well known Birkhoff theorem, provides a representation of general BL-rings.

We recall that a commutative integral residuated lattice can be defined as a nonempty set $L$ with four binary operations
 $\wedge, \vee, \otimes, \to$, and two constants $0,1$ satisfying: (i) $\mathbb{L}(L):=(L,\wedge, \vee, 0, 1)$ is a bounded lattice; (ii) $(L,\otimes, 1)$ is a commutative  monoid; and (iii) $(\otimes, \to)$ form an adjunct pair, i.e., $x\otimes y\leq z$ iff $x\leq y\to z$, for all $x, y, z\in L$.
 
An \textit{R$\ell$ monoid} is a residuated lattice  $L$ satisfying the divisibility axiom, that is: $ x\wedge y=(x \to y)\otimes x$, for all $x, y\in L$.

A \textit{BL- algebra} is an \textit{R$\ell$ monoid} $L$ satisfying the pre-linearity axiom, that is: $(x\to y)\vee (y\to x)=1$, for all $x, y\in L$

An \textit{MV-algebra} is a BL-algebra  $L$ satisfying the double-negation law: $x^{\ast \ast}=x$, for all $x\in L$, where $x^\ast=x\to 0$. \\

Given any commutative ring $R$ generated by idempotents (that is for every $x\in R$, there exists an idempotent $e\in R$ such that $ex=x$), the lattice of ideals of $R$ form a residuated lattice $A(R):=\langle \text{Id}(R), \wedge, \vee, \otimes, \to, \{0\}, R\rangle $, where $I\wedge J = I\cap J$,
$I\vee J=I+J$, $I \otimes J:=I\cdot J
$, $I \to J := \{x\in R: xI\subseteq J \}$.\\
Note that $I^\ast$ is simply the annihilator of $I$ in $R$.

The following notations will be used throughout the paper. Given a commutative ring $R$, recall that an ideal $I$ of $R$ is called an annihilator ideal (resp. a dense ideal) if $I=J^\ast$ for some ideal $J$ of $R$ (resp. $I^\ast=0$).

$A(R)$ denotes the residuated lattice of ideals of $R$;

$MV(R)$ denotes the set of annihilator ideals of $R$;

$D(R)$ denotes the set of dense ideals of $R$. 
\section{BL-rings, definitions, examples and first properties}
In this section, we introduce the notion of BL-rings. As announced, these should be rings $R$ for which $A(R)$ is naturally equipped with a BL-algebra structure. Some of the main properties of these rings, and their connections to other known classes of rings are established.
\begin{defi}
A commutative ring $R$ is called a BL-ring if for all ideals $I, J$ of $R$,\\
BLR-1: $I\cap J=I\cdot (I\to J)$, \\
BLR-2: $(I\to J)+(J\to I)=R$.
\end{defi}
Note that BLR-1 is equivalent to $I\cap J\subseteq I\cdot (I\to J)$ since the inclusion $I\cdot (I\to J) \subseteq I\cap J$ holds in any ring. In addition, BLR-2 is easily seen to be equivalent to each of the following conditions:
\begin{itemize}
\item[BLR-2.1:] $(I\cap J)\to K=(I\to K)+(J\to K)$ for all ideals $I, J, K$ of $R$.
\item[BLR-2.2:] $I\to (J+K)=(I\to J)+ (I\to K)$  for all ideals $I, J, K$ of $R$.
\end{itemize}
Recall \cite{Grif} that a commutative ring is called a multiplication ring if every ideals $I, J$ of $R$ such that $I\subseteq J$, there exists an ideal $K$ of $R$ such that $I=J\cdot K$.
\begin{prop}\label{rl=mu}
A commutative ring satisfies BLR-1 if and only if it is a multiplication ring.
\end{prop}
\begin{proof}
Suppose that $R$ is a BL-ring and let $I, J$ be ideals of $R$ such that $I\subseteq J$. Then by BLR-1, $I=I\cap J=I\cdot (I\to J)$. Take $K=I\to J$.\\
Conversely, suppose that $R$ is a multiplication ring and let $I, J$ be ideals of $R$. Since $I\cap J\subseteq I$, there exists an ideal $K$ of $R$ such that $I\cap J=I\cdot K$. Hence, $I\cdot K\subseteq J$ and it follows that $K\subseteq I\to J$. Thus, $I\cap J\subseteq I\cdot (I\to J)$. As observed above, the inclusion $I\cdot (I\to J) \subseteq I\cap J$ holds in any ring. Therefore, BLR-1 holds as needed.
\end{proof}
Recall that a ring $R$ is said to be generated by idempotents if $\; \; \text{for every}\; \;  x\in R,\; \text{there exist}\; \;  e=e^2\in R\; \;  \text{such that}\; \; xe=x$.
\begin{cor} \label{idempotent}
\begin{itemize}
\item[1.] Every BL-ring is generated by idempotents.
\item[2.] A commutative ring is a BL-ring if and only if $A(R)$ is a BL-algebra. 
\end{itemize}
\end{cor}
\begin{proof}
1. BLR-1 implies that $R$ is a multiplication ring, and it is known that every multiplication ring satisfies the condition \cite[Cor. 7]{Grif}.\\
2. This is clear from (1) and the axioms BLR-1 and BLR-2.
\end{proof}
\begin{ex}\label{exam}
1. Discrete valuation rings (dvr).\\
 The ideals of a dvr are principal and totally ordered by the inclusion. Clearly BLR-2 holds in any chain ring. As for BLR-1, let $I, J$ be ideals of a dvr $R$. If $I\subseteq J$, then $I\to J=R$ and $I\cap J=I\cdot (I\to J)$. On the other hand, if $J\subseteq I$, since $I$ is principal, then $I=aR$ for some $a\in R$. Let $j\in J\subseteq aR$, then $j=ax$. So $ax\in J$ and $x\in I\to J$. Hence, $j\in I\cdot (I\to J)$ and $I\cap J\subseteq I\cdot (I\to J)$.\\
2. Noetherian multiplication rings. \\
By Proposition \ref{rl=mu}, every multiplication ring satisfies BLR-1. In addition, if $R$ is a Noetherian multiplication ring, then $R$ is a Noetherian arithmetical ring and by \cite[Thm. 3]{Jen}, $K\to (I+J)=(K\to I)+(K\to J)$ for all ideals $I, J, K$ of $R$. Hence, $R$ satisfies BLR-2.2, which is equivalent to BLR-2 as observed earlier. Whence, $R$ is a BL-ring as claimed. \\
3. \L ukasiewicz rings. Indeed, if $R$ is a \L ukasiewicz ring, then its ideals form an MV-algebra \cite{BN}.\\
4. G\"{o}del rings. Indeed, if $R$ is a G\"{o}del ring, its ideals from a a BL-algebra in which $I\cdot J=I\cap J$ \cite{BNM}.
\end{ex}
\begin{rem}\label{obs}
\begin{itemize}
\item[1.] Each of the classes of rings of Example \ref{exam} is a proper subclass of BL-rings. In fact $\mathbb{Z}$ is a Noetherian multiplication ring that is neither a dvr nor a \L ukasiewicz ring. In addition, $\oplus_{n=1}^\infty\mathbb{R}$ is a \L ukasiewicz ring that is neither Noetherian, nor a dvr.
\item[2.] A Noetherian ring is a BL-ring if and only if it is a multiplication ring. In addition, note that Noetherian multiplication rings are ZPI-rings (see for e.g., \cite[Exercise 10(b) pp. 224]{Lars}), and ZPI-rings are direct sum of a finite number of Dedekind domains and special primary rings. Therefore, Noetherian BL-rings are direct sum of a finite number of Dedekind domains and special primary rings.
\end{itemize}
\end{rem}
Recall that a Baer ring is a ring in which every annihilator ideal is generated by an idempotent, i.e., for every ideal $I$ of $R$, there exists an idempotent $e\in R$ such that $I^\ast=eR$. \par
The following lemma will be needed when working with BL-rings.
\begin{lem}\label{quotient}Let $R$ be a ring, and $I, J, K$ be
ideals of $R$ such that $I\subseteq J, K$. Then,\\
(a) $I\subseteq (I^\ast \cdot J)^\ast, J\to I, J\to K, K\to J$;\\
(b) $(J/I)^\ast=(J\to I)/I$;\\
(c) $(J/I)\to (K/I)=(J\to K)/I$.
 \end{lem}
\begin{proof} These are easily derived from the definitions of the operations involved.
\end{proof}
Note that if a ring satisfies BLR-2, then since $I\to J=I\to (I\cap J)$, it must satisfy the following.\\
BLR-3: $I\cap J=0$ implies $I^\ast +J^\ast=R$.
\begin{prop}\label{blr23}
A ring $R$ satisfies BLR-2 if and only if every quotient (by an ideal) of $R$ satisfies BLR-3.
\end{prop}
\begin{proof}
Suppose that $R$ satisfies BLR-2, and let $I$ be an ideal of $R$. Let $I\subseteq J, K$ such that $(J/I)\cap (K/I)=I$. Then, $J\cap K=I$. Now, $(J/I)^\ast+(K/I)^\ast=(J\to I)/I+(K\to I)/I=(J\to (J\cap K))/I+(K\to (J\cap K))/I=((J\to K)+(K\to J))/I=R/I$. Thus, $R/I$ satisfies BLR-3. Conversely, suppose that every factor of $R$ satisfies BLR-3. Let $I, J$ be ideals of $R$, then $R/(I\cap J)$ satisfies BLR-3. Since  $(I/(I\cap J))\cap (J/(I\cap J))=I\cap J$, then $(I/(I\cap J))^\ast+(J/(I\cap J))^\ast=R/(I\cap J)$. That is, $(I\to (I\cap J))/I+(J\to (I\cap J))/I=R/(I\cap J)$, or $(I\to J)/(I\cap J)+(J\to I)/(I\cap J)=R/(I\cap J)$. Thus, $((I\to J)+(J\to I))/(I\cap J)=R/(I\cap J)$ and it follows that $(I\to J)+(J\to I)=R$. So, $R$ satisfies BLR-2 as needed.
\end{proof}
\begin{cor}
Every multiplication ring satisfies BLR-2 if and only if every multiplication ring satisfies BLR-3.
\end{cor}
The following example, which is a special case of \cite[Ex. 4]{Grif} is a multiplication ring that does not satisfy BLR-2.
\begin{ex}\label{notbl}
Let $F$ be any field and let $R$ be the subring of $\prod_{k=1}^\infty F$ generated by $\oplus_{k=1}^\infty F$ and the constant functions from $\mathbb{N}\to F$.\\
Then $R$ is a multiplication ring with identity.\\
Note that $R$ is the subring of $\prod_{k=1}^\infty F$ of all sequences $\mathbb{N}\to F$ that are eventually constant. That is $f\in R$ if and only if there exists $x\in F$ and $n\geq 1$ such that $f(k)=x$ for all $k\geq n$.\\
Now, let $I=\{f\in R: f(2k)=0\; \text{for all}\; k\in \mathbb{N}\}$ and $J=\{f\in R: f(2k+1)=0\; \text{for all}\; k\in \mathbb{N}\}$.\\
Then $I, J\subseteq \oplus_{k=1}^\infty F$, $I\cap J=0$, $I^\ast=J$, $J^\ast=I$. Thus, $I^\ast +J^\ast=I+J\ne R$.
\end{ex}
\begin{prop}\label{NP}
Let $R$ be a ring that is generated by idempotents and $P$ be a prime ideal of $R$. Recall \cite[\S IX.4]{Lars} that $$N(P)=\{x\in R:xs=0\; \; \text{for some }\; \; s\in R\setminus P\}$$
Then $\displaystyle\bigcap_{P} N(P)=0$.
\end{prop}
\begin{proof}
Let $x\ne 0$, then $(xR)^\ast\ne R$. Thus, there exists a prime ideal $P$ of $R$ such that $(xR)^\ast \subseteq P$. We claim that $x\notin N(P)$. Indeed, if $x\in N(P)$, then $xs=0$ for some $s\notin P$. This would imply that $s\in (xR)^\ast \subseteq P$, which is a contradiction.
\end{proof}
\begin{prop}
Every multiplication ring with unity is a subring of a direct product of dvrs and SPIRs.
\end{prop}
\begin{proof} Let $R$ be multiplication ring with unity.
Consider $\varphi :R\to \prod_PR_P$ defined by $\varphi(x)=(\frac{x}{1})_P$.
Then $\varphi$ is a ring homomorphism and ker$\varphi=\bigcap_{P} N(P)$. Thus, $\varphi$ is injective by Proposition \ref{NP}. On the other hand, $R$ is an AM-ring \cite[Lemma 2.4]{BP}. Therefore, by \cite [Thm. 9.23, Prop. 9.25, Prop. 9.26]{Lars}, each $R_P$ is either a dvr or an SPIR.
\end{proof}
\begin{prop}\label{closed}BL-rings are closed under each of the following operations.\\
1. Finite direct products;\\
2. Arbitrary direct sums;\\ 
3. Homomorphic images.
\end{prop}
\begin{proof}
1. Let $R=\prod_{k=1}^nR_k$, where each $R_k$ is a BL-ring. Using the fact that each $R_k$ is generated by idempotents, one gets that any ideal $I$ of $R$ is of the form $I=\prod_{k=1}^nI_k$, where $I_k$ is an ideal of  $R_k$ for all $k$. \\
On the other hand, if $I=\prod_{k=1}^nI_k$ and $J=\prod_{k=1}^nJ_k$, one can easily verify the following identities: 
$I\cdot J=\prod_{k=1}^nI_k\cdot J_k$, $I\to J=\prod_{k=1}^nI_k\to J_k$, $I\cap J=\prod_{k=1}^nI_k\cap J_k$, and $I+J=\prod_{k=1}^nI_k+J_k$.\\
From these identities, it becomes clear that $R$ satisfies BLR-1 and BLR-2 since each $R_k$ does. Therefore, $R$ is a BL-ring.\\
2. This is very similar to the above. Indeed, one proves that ideals of direct sums of BL-rings are direct sums of ideals and the argument goes through as in (1), with each instance of the finite direct product replaced by a direct sum.\\
3. Let $R$ be a BL-ring and $I$ be an ideal of $R$. We shall show that $R/I$ is a BL-ring. Recall that ideals of $R/I$ are of the form $J/I$, where $J$ is an ideal of $R$ with $I\subseteq J$. Let $J, K$ be ideals of $R$ containing $I$. Using the properties stated in Proposition \ref{quotient}, we have $(J/I)\cap (K/I)=(J\cap K)/I=(J\cdot (J\to K))/I=(J/I)\cdot (J\to K)/I=(J/I)\cdot ((J/I)\to (K/I))$. Thus, $R/I$ satisfies BLR-1. The verification of BLR-2 is similar. Therefore, $R/I$ is a BL-ring as needed.
\end{proof}
The following examples shows that an arbitrary direct product of BL-rings needs not a BL-ring.
\begin{ex}

\end{ex}
\section{Connection with Baer rings and Von Neumann ring}
\begin{prop}\label{Baer}
A reduced ring with identity satisfies BLR-3 if and only if it is a Baer ring.
\end{prop}
\begin{proof}
Suppose $R$ is a reduced ring with identity that satisfies BLR-3. Let $I$ be an ideal of $R$, then since $R$ is reduced, $I\cap I^\ast =0$. It follows from BLR-3 that $I^\ast+I^{\ast\ast}=R$. Hence, $1=a+b$ for some $a\in I^\ast$ and $b\in I^{\ast\ast}$. Thus, $a=a.1=a(a+b)=a^2+ab=a^2$ and $a$ is idempotent. Now, for every $x\in I^\ast$, $x=x.1=x(a+b)=xa+xb=xa$. Therefore, $I^\ast=aR$ and $R$ is a Baer ring.\par
Conversely, suppose that $R$ is a Baer ring and let $I, J$ be ideals such that $I\cap J=0$. Then $I\subseteq J^\ast=eR$, for some idempotent $e\in R$. Note that since $I\subseteq eR$ and $e$ is idempotent, then $1-e\in I^\ast$. Thus, $1=(1-e)+e\in I^\ast+I^{\ast\ast}$ and $I^\ast+I^{\ast\ast}=R$.
\end{proof}
Recall that in every Baer ring $R$, $a^\ast$ is the unique idempotent element in $R$ such that $(aR)^\ast=a^\ast R$. An ideal $I$ of a Baer ring is called a Baer-ideal if for every $a, b\in R$ such that $a-b\in I$, then $a^\ast-b^\ast\in I$. 
\begin{cor}
For every Baer ring $R$ and every Baer-ideals $I, J$ of $R$, $(I\to J)+(J\to I)=R$.
\end{cor}
\begin{proof}
Let $R$ be a Baer ring and $I, J$ be Baer-ideals of $R$, then $I\cap J$ is a Baer-ideal. Hence, $R/(I\cap J)$ is a Baer ring and by Proposition \ref{Baer}, satisfies BLR-3. It follows as in the proof of Proposition \ref{blr23}, that $(I\to J)+(J\to I)=R$.
\end{proof}
\begin{prop}\label{quo-m}
1. Every quotient (by an ideal) of a multiplication ring is a multiplication ring.\\
2. Every quotient (by an Baer-ideal) of a Baer ring is a multiplication Baer ring.
\end{prop}
\begin{proof}
1. Let $R$ be a multiplication ring and $I$ an ideal of $R$. Let $J/I\subseteq K/I$ be ideals of $R/I$, then $J\subseteq K$. Thus, as $R$ is a multiplication ring, there exists an ideal $T$ of $R$ such that $J=K\cdot T$. Hence, $J/I=(K\cdot T)/I=K/I\cdot T/I$. Therefore, $R/I$ is a multiplication ring.\\
2. \cite[Lemma 3]{speed1}
\end{proof}

Recall that a Commutative ring with is a Von Neumann ring (VNR) if and only if $R_P$ is a field for all prime ideals of $R$.\\
\begin{prop} \label{vnr}
Every VNR is a multiplication ring.
\end{prop}
\begin{proof}
Note that by Proposition \ref{rl=mu}, we simply have to prove BLR-1, which we shall do locally. Let $R$ be a VNR, and $I, J$ be ideals of $R$, and $P$ a prime ideal of $R$. We need to show that $I_P\cap J_P=I_P\cdot (I\to J)_P$. Since $R_P$ is a field, then $I_P=0$ or $R_P$ and $J_P=0$ or $R_P$. The equation is obvious when $I_P=0$. We consider the remaining two cases. \\
Case 1: If $I_P=J_P=R_P$, then $J\cap (R\setminus P)\ne \emptyset$. But, $J\subseteq I\to J$, so $J\cap (R\setminus P)\subseteq (I\to J)\cap (R\setminus P)$. Hence $(I\to J)\cap (R\setminus P)\ne \emptyset$ and $(I\to J)_P=R_P$. Thus $I_P\cdot (I\to J)_P=R_P\cdot (I\to J)_P=(I\to J)_P=R_P=I_P\cap J_P$.\\
Case 2: Suppose $I_P=R_P$ and $J_P=0$, then $I\cap (R\setminus P)\ne \emptyset$. We need to show that $(I\to J)_P=0$. Since $I\cap (R\setminus P)\ne \emptyset$, there exists $s\in I$ and $s\notin P$. Now, let $x\in (I\to J)$ and $t\notin P$. Then, $xs\in J$ and since $J_P=0$, it follows that $xs/t=0/t$. Thus, $xst'=0$ for some $t'\notin P$, which implies that $x/t=0/t$. Whence, $(I\to J)_P=0$ as needed.
\end{proof}
It follows from Proposition \ref{vnr} and its proof that a VNR is a BL-ring if and only if it satisfies for all ideals $I, J$ and all prime ideal $P$, $$I_P=J_P=0\; \; \text{implies}\; \; (I\to J)_P=R_P$$

\section{representation and further properties of BL-rings}
Recall that it follows from the most celebrated Birkhoff subdirectly irreducible representation Theorem (see for e.g., \cite[Thm.8.6]{BS}), every BL-ring $R$ is a subdirect product of subdirectly irreducible rings, all of whom are homomorphic images of $R$. It follows from Proposition \ref{closed} that every BL-ring $R$ is a subdirect product of subdirectly irreducible BL-rings. This justifies the need to start our analysis with subdirectly irreducible BL-rings. It is known that for every subdirectly irreducible commutative ring $R$ with minimal ideal $M$ is either $R$ is a field or $M^2=0$ (see for e.g., \cite{Fe}).
\begin{prop}\label{subirr}
Let $R$ be a subdirectly irreducible BL-ring with minimal ideal $M$. Then
\begin{itemize}
\item[1.] The annihilator ideals of $R$ are linearly ordered and finite in number;
\item[2.] Every ideal of $R$ is either an annihilator ideal or dense;
\item[3.] $M$ is an annihilator ideal;
\item[4.] For every annihilator ideal $I\ne R$ and every dense ideal $J$, $I\subseteq J$ and $J\to I=I$;
\item[5.] For every ideals $I, J$ of $R$, either $I\to J$ or $J\to I$ is dense.
\end{itemize}
\end{prop}
\begin{proof}
Note that $A(R)$ is a BL-algebra and $MV(R)$ is an MV-algebra, more precisely the MV-center of $A(R)$. Moreover, $(\bigvee I)^\ast=\bigcap I^\ast$, which implies that every subset of $MV(R)$ has an infimum. It follows from this that every subset of $MV(R)$ also has a supremum since $\bigwedge S=(\bigvee S^\ast)^\ast$ \cite[Lemma 6.6.3]{C2}. Thus, $MV(R)$ is a complete MV-algebra. In addition, for every nonzero ideal $I$ of $R$, $M\subseteq I$, so $I^\ast \subseteq M^\ast$. Therefore, for every proper ideal $J$ in  $MV(R)$, we have $J\subseteq M^\ast$. This means that $MV(R)\setminus \{R\}$ has a maximum element, namely $M^\ast$ (note that $M\ast\ne R$ since $M\ne 0$). To see that $(MV(R), \subseteq)$ is a chain, let $X, Y\in MV(R)$ such that $X\nsubseteq Y$ and $Y\nsubseteq X$, then $X\to Y, Y\to X\ne R$ and $X\to Y, Y\to X\subseteq M^\ast$. Hence, by the pre-linearity axiom, $R=(X\to Y)\vee(Y\to X)\subseteq M^\ast$. Hence, $M^\ast =R$, which is a contradiction. Therefore, $(MV(R), \subseteq)$ is an MV-chain as claimed. But the only complete MV-chains are finite \L ukasiewicz chains and $[0,1]$. The condition $MV(R)\setminus \{R\}$ has a maximum element implies that $MV(R)$ is a finite \L ukasiewicz chain. This completes the proof of (1) and (3).

On the other hand, since the MV-center of $A(R)$ is a finite \L ukasiewicz chain, it follows from \cite[Remark 3.3.2]{BM} that $A(R)\cong MV(R)\oplus D(R)$, the ordinal sum of $MV(R)$ and $D(R)$. Readers unfamiliar with the ordinal sums of hoops may \cite[\S 3.1]{BM} for basic definitions and properties. The property stated in (2) clearly holds in $MV(R)\oplus D(R)$, and therefore in $A(R)$. Similarly, (4) follows from the definitions of the implication $\to$ in the ordinal sum.

It remains to prove (5). Since $A(R)$ is a BL-algebra, it is known (see for e.g., \cite[p. 368]{BM}) that the set $D(L)$ of dense elements of any BL-algebra $L$ is an implicative filter and $L/D(R)\cong MV(L)$. Therefore, $A(R)/D(R)\cong MV(R)$, and since $MV(R)$ is linearly ordered, we deduce that $A(R)/D(R)$ is linearly ordered. The conclusion now follows from the definitions of order and $\to$ on $A(R)/D(R)$. 

\end{proof}
The following result, which is the analog of \cite[Theorem 3.9]{BNM} shields light on the structure of a general BL-ring.
\begin{thm}(A Representation Theorem for BL-rings)
Every BL-ring $R$ is a subdirect product of a family $\{R_t:t\in T\}$ of subdirectly irreducible BL-rings satisfying:
\begin{itemize}
\item[1.] $A(R_t)\cong MV(R_t)\oplus D(R_t)$ for all $t\in T$;
\item[2.] $A(R)$ is a subdirect product of $\{A(R_t):t\in T\}$;
\item[3.] $A(R_t)$ is a BL-algebra with a unique atom.
\end{itemize}
\end{thm}
\begin{proof}
In light of the opening remarks of this section, more importantly Proposition \ref{subirr} and its proof, we only need to prove (2). To show that $A(R)$ is a subdirect product of $\{A(R_t):t\in T\}$, recall that there is a family $\{I_t:t\in T\}$  ideals of $R$ such that $R_t=R/I_t$ for all $t$ and $\bigcap I_t=0$. Now, define $\Theta: A(R)\to \prod_{t\in T}A(R/I_t)$ by $\Theta(I)=I+I_t \mod I_t$. It is readily verified that $\Theta$ is a subdirect embedding of BL-algebras.
\end{proof}
We shall end our study by establishing some (further) properties of BL-rings.

\begin{prop} Let $R$ be a BL-ring. Then
\begin{enumerate}
\item[(i)] Let $I$ be an ideal of $R$ and $P$ a prime ideal of $R$. Then $I\subseteq P$ or $I\rightarrow P=P$. 
\item[(ii)] Every proper ideal is contained in a prime ideal of $R$.
\item[(iii)] If $P,Q\subseteq R$ are prime ideals that are not comparable, then they are comaximal, that is, $P+Q=R$.
\item[(iv)] If $P,Q\subseteq R$ are distinct minimal prime ideals, then they are comaximal.
\item[(v)] Let $I$ be an ideal of $R$. Then the prime ideals below $I$, if any, form a chain.
\item[(vi)] Suppose $R$ is a local ring. Then the prime ideals of $R$ form a chain and each ideal is a power of a prime ideal.
\item[(vii)] Suppose $P,Q\subseteq R$ are prime ideals that are not comparable. Then there is a BL-epimorphism from $A(R)\rightarrow A(R/p\oplus R/Q)$. 
\item[(viii)] Suppose the set of minimal primes Min($R$) is finite, then $R$ is a finite direct sum of Dedekind domains and special primary rings.
\end{enumerate}
\end{prop}
\begin{proof}
\begin{enumerate}
\item[(i)] Let $I$ be an ideal of $R$ and $P$ a prime ideal of $R$. We have from BL-1 $I\cap P=I(I\rightarrow P)$. So $I(I\rightarrow P)\subseteq P$. Thus $I\subseteq P$ or $I\rightarrow P\subseteq P$. Clearly $P\subseteq I\rightarrow P$ and the result follows. 
\item[(ii)] This holds since every BL-ring is generated by idempotents (Corollary \ref{idempotent} 1.). In fact, Let $I$ be a proper ideal $R$ and let $a\notin I$. There is an $e\in R$ with $e^2=e$ such that $a=ae$. It is clear that $e\notin I$ since $a\notin I$. Let $P$ be a maximal ideal with respect containing $I$ but not containing $e$. Suppose now that $xRy\subseteq P$ where $x,y\notin P$. then $e\in P+RxR$ and  $e\in P+RyR$ (as if $e$ is not in the ideal it would contradict the maximality of $P$). It follows that $e\in P+xRyR\subseteq P$, which is a contradiction. Thus $P$ is prime. 
\item[(iii)] This follows from the combination of BL-2 and $(i)$, since $R=(P\rightarrow Q)+(Q\rightarrow P)=P+Q=R$.
\item[(iv)] This follows from the fact that distinct minimal primes are not comparable and the use of $(iii)$.
\item[(v)] Suppose $P,Q\subseteq I$ with $P$ and $Q$ prime ideals. If $P$ and $Q$ are not comparable, we have $R=P+Q\subseteq R$ which contradicts the fact that $I$ is proper. Hence the prime ideals below $I$, if any, form a chain.
\item[(vi)] Suppose $R$ is a local ring. We have a unique maximal ideal that contains all the prime ideals of $R$. Thus the prime ideals of $R$ form a chain by $(v)$. Now the radical of an ideal $I$ of $R$ is the intersection of all prime ideals containing $I$. Since the prime ideals form a chain, the intersection of a chain of prime ideals is a prime ideal and it then follows that the radical of $I$ is a prime ideal. Since $R$ is a multiplication ring, it follows that $I$ is a power of a prime ideal \cite[Theorem 5.]{Mott} 
\item[(vii)] Suppose $P,Q\subseteq R$ are prime ideals that are not comparable. We know that $R/P$ and $R/Q$ are BL-algebras (so $R/p\oplus R/Q$ is also a BL-algebra) and $P+Q=R$. Also, $R^2+P=R^2+Q=R$. By the Chinese Remainder Theorem the natural map $R\rightarrow R/p\oplus R/Q$ is onto. Thus the natural map induces naturally a BL-algebra epimorphism $A(R)\rightarrow A(R/p\oplus R/Q)$. 
\item[(viii)] This holds since $R$ is a multiplication ring and \cite[Theorem 11.]{Mott}. 
\end{enumerate}
\end{proof}

\textbf{Acknowledgments:} Thanks are due to Professor Bruce Olberding for the fruitful discussion we had with him over multiplication, Baer, and Von Neumann rings.

\end{document}